\documentclass[14pt]{article}
= 0.0in \headheight = 0.0in \topmargin = 0.2in
\def\Dj{\hbox{D\kern-.73em\raise.30ex\hbox{-}
\raise-.30ex\hbox{}}}
\def\dj{\hbox{d\kern-.33em\raise.80ex\hbox{-}
\raise-.80ex\hbox{\kern-.40em}}}

\usepackage{amssymb}
\usepackage{color}
\usepackage{amssymb}
\usepackage{float}
\usepackage{amsmath}
\usepackage{caption}
\usepackage{graphicx}
\usepackage{latexsym}
\usepackage{amsthm}
\usepackage{tikz}
\usepackage{lineno}
\usepackage{enumitem}
\usepackage{epsfig}
\usepackage{graphicx}
\usepackage{amsmath,amsthm,amsfonts,amssymb,amscd,cite}
\allowdisplaybreaks
\usepackage{mathrsfs}

\newtheorem{theorem}{Theorem}[section]
\newtheorem{lemma}[theorem]{Lemma}

\newtheorem{definition}[theorem]{Definition}

\newtheorem{example}[theorem]{Example}

\begin{document}

%

%
%

\title{Efficient Enumeration of Cliques in Graphs with Bounded Maximum Degree}
\author{Shi-Cai Gong\thanks{Corresponding author.}\thanks{ Supported by National Natural Science Foundation of
China(12271484).}, Jia-Jin Wang, Xin-Hao Zhu and Bo-Jun Yuan
\thanks{ Supported by National Natural Science Foundation of
China(12201559).}~\thanks{ E-mail addresses: scgong@zafu.edu.cn(S. Gong), wangjiajin0876@163.com(J. Wang), zhuxinhao0508@163.com(X. Zhu) and
ybjmath@163.com(B. Yuan).}
\\
\\{\small \it  School of Science, Zhejiang University of Science and Technology, }\\{\small \it
Hangzhou, 310023, P. R. China}
   }
\date{}
\maketitle

\baselineskip=0.20in

\noindent {\bf Abstract.}
In recent years, there has been a surge of interest in extremal problems concerning the enumeration
 of independent sets or cliques in graphs with specific constraints. For instance, the Kahn-Zhao
 theorem establishes an upper bound on the number of independent sets in a $d$-regular graph.
 Building on this, Cutler and Radcliffe extended the result by identifying the graph that maximizes
 the number of cliques among graphs with  bounded order and maximum degree. 
\vspace{2mm}

In this paper, we introduce an innovative approach for counting cliques
 in graphs with a bounded maximum degree. To demonstrate the effectiveness of the method, we provide a new proof  for  the above
  Cutler-Radcliffe theorem and the Kahn-Zhao
 theorem.\vspace{3mm} 

%



\baselineskip=0.25in

\section{Introduction}

There has been significant interest in extremal problems involving
counting independent sets or cliques in graphs with specific constraints.
The study of clique enumeration can be traced back to 1941 when Tur\'{a}n \cite{Tur}
first bounded the number of $2$-cliques in graphs with fixed number of vertices and bounded clique number.
As a natural generalization of Tur\'{a}n Theorem, Zykov \cite{z} first determined the maximum number of cliques of a fixed order in a graph with fixed number
of vertices and bounded clique number
 (see also \cite{er,ha,ro,sa}).
Since then, numerous researchers
have contributed to this field, developing a rich body of literature that explores various aspects of clique and independent set enumeration.
Another notable example is the Kahn-Zhao theorem, which was initially
proved for bipartite graphs by Kahn \cite{Kahn} and later extended to all graphs by Zhao \cite{Zhao}.\vspace{3mm}

\begin{theorem} {\em (Kahn-Zhao)} \label{main0}
  If $G$ is a $d$-regular graph with $n$ vertices, then
  \begin{align*}
    i(G)^{\frac{1}{n}}\leq i(K_{d,d})^{\frac{1}{2d}}=(2^{d+1}-1)^{\frac{1}{2d}},
  \end{align*}
  where $i(G)$ denotes the number of independent sets contained in the graph $G$.
\end{theorem}

If $G$ is a graph with $t$ independent sets, then $\overline{G}$, the complement of $G$, is a graph with $t$ cliques. Any degree condition
on $G$ can be translated into a corresponding degree condition on $\overline{G}$. Therefore,  extremal enumeration problems for independent
sets run in parallel to those for cliques.
In 2014, Cutler and Radcliffe \cite{cr} extended the conditions for clique counting problems from regularity to maximum degree.
They resolved and strengthened the conjecture posed by Galvin \cite{Galvin} as follows:\vspace{2mm}

\begin{theorem}{\em (Cutler and Radcliffe)}\label{maina}
  For all $n$, $r\in \mathbb{N}$, write $n=a(r+1)+b$ with $0\leq b \leq r$. If $G$ is a graph with order $n$  and maximum degree at most $ r$, then
  \begin{align*}
    k(G)\leq k(aK_{r+1}\cup K_b),
  \end{align*}
  with equality if and only if $G=aK_{r+1}\cup K_b$, or $r=2$ and $G=(a-1)K_{3}\cup C_4$ or $(a-1)K_3\cup C_5$,
  where $k(G)$ denotes the number of cliques contained in the graph $G$.
\end{theorem}

Both the Kahn-Zhao theorem and the Cutler-Radcliffe theorem have been used as a foundation for further research in extremal graph theory.
 For example, it has been applied to study the supersaturation of subgraph counts,
 where researchers investigate how the number of certain subgraphs grows as the number of edges in a graph increases.
 Bounds on the number of cliques in a graph have recently been applied in the analysis of an
algorithm for finding small separators \cite{32} and in the enumeration of minor-closed
families \cite{28}.
Additionally, these two theorems have been further extended to incorporate other constraint conditions.
 For more in-depth research on subgraph counting problems, one can refer to \cite{AMM,ch,ckk,cnr,kra}
 and the references therein.\vspace{2mm}


In this paper, we introduce an innovative approach for counting cliques
 in graphs with a bounded maximum degree. To demonstrate the effectiveness of the method, we provide a new proof  for  the above
  Cutler-Radcliffe theorem and the Kahn-Zhao
 theorem.
  The rest of the paper is  organized as follows.
In Section $2$, we make some necessary preparations, introduce the concepts of the clique
 weight of a vertex and the clique weight of an edge, as well as some related lemmas.
 Then we  provide a new proof  for  Theorem \ref{maina}, as well as Theorem \ref{main0}, in Section $3$.\vspace{2mm}

\section{Preliminaries}
All graphs considered in this paper are   finite, simple and undirected.
For undefined graph terminology and notation, readers are referred to the monograph \cite{Bondy}.\vspace{2mm}

Let $G=(V(G),E(G))$ be a  graph with $v\in V(G)$. Denote
by $N_G(v)$ and $N_G[v]$ the open and closed neighborhoods of $v$ in $G$, respectively.
Moreover, we use  $d_H(v)$ to represent the degree of $v$ in a subgraph $H$, i.e., $d_H(v)=|N_H(v)|$.
 Denote by $\Delta(G)$ the maximum degree of the graph $G$.
For a subset $S$ of
$V(G)$, the subgraph of $G$ induced on $S$ is denoted by $G[S]$.
For simplicity, we write $G[V(G)\backslash S]$ as $G-S$. For a subgraph $H$ of $G$,
we use $G-H$ to denote the subgraph obtained from $G$ by removing only the edges in $H$
 while retaining the vertices in $H$. Additionally, we use $aH$ to represent $a$ copies of the graph $H$.
 We use $K_r-tP_2$ to denote the graph obtained from the complete graph $K_r$ by deleting a $t$-matching, where $t\le \lfloor\frac{r}{2}\rfloor$. \vspace{2mm}

Let $G_1=(V(G_1),E(G_1))$ and $G_2=(V(G_2),E(G_2))$
 be two graphs. The {\it union} of $G_1$ and $G_2$, denoted by $G_1\cup G_2$,
 is defined as the graph with vertex set $V(G_1)\cup V(G_2)$ and edge set $E(G_1)\cup E(G_2)$.
In particular, if $G_1$ and $G_2$ are disjoint (i.e., $V(G_1)\cap V(G_2)=\emptyset$),
the union is  referred to as the {\it disjoint} union of the two graphs.\vspace{2mm}

Set $[n]:=\{1,2,\ldots,n\}$. The {\it colexicographic order}, denoted by $<_C$, is defined on finite subsets of $\mathbb{N}$ as follows: for two subsets $A$ and $B$, we have $A <_C B$ if the maximum element in the symmetric difference $A \Delta B$ belongs to $B$.
When restricted to $2$-subsets of $[n]$, the colexicographic order induces a specific ordering on the edges of the complete graph $K_n$. The first few edges in this colex ordering are:
$$\{1, 2\}, \{1, 3\}, \{2, 3\}, \{1, 4\}, \{2, 4\}, \{3, 4\}, \{1, 5\}, \{2, 5\},\{3, 5\},\{4, 5\}, \ldots.$$
The {\it colex graph} $\mathcal{C}(n,m)$ is defined as the graph with vertex set $[n]$ and edge set consisting of the first $m$ edges in the colex order on $E(K_n)$.
For more research on colex graph, one can refer to \cite{aanw,kr}. It is worth noting that when $m = \binom{r}{2} + s$ with $0 \leq s < r$, the graph $\mathcal{C}(n,m)$ comprises a clique of order $r$ together with an additional vertex connected to $s$ vertices of this clique.\vspace{2mm}

Let \( G \) be a graph with \( u \in V(G) \). We denote by \( k_t(G) \)
 the number of cliques of size \( t \) in \( G \),
 and the total number of cliques is given by \( k(G) = \sum_{t \geq 0} k_t(G) \).
  The notation  \( k_t(u;G) \) is used to represent the number of cliques
   of size \( t \) in \( G \) that contain the vertex \( u \).
   For brevity, we write \( k_t(u) \) instead of \( k_t(u;G) \)
   when the graph \( G \) is clear from the context. Kruskal-Katona theorem implies the following:\vspace{2mm}

\begin{theorem} {\em \cite{kat,kru}.} \label{lemma1} For any positive integers $t, m$, and any graph $G$ with $n$ vertices and $m$ edges,
	we have that $$k_t(G) \le k_t(\mathcal{C}(n,m)).$$ Moreover, $\mathcal{C}(n,m)$
	is the unique graph satisfying the equality
	when $s \ge t-1$, where $m = \binom{r}{2}
	+ s$ with $0 \le s <r$.
\end{theorem}

Given that each $t$-clique is composed of $t$ vertices with equal contribution,
 we introduce the concept of {\it vertex clique weight}. This concept enables
  us to express the global parameter $k(G)$ as the sum of local vertex contributions,
  where each vertex contributes $1/t$ to each $t$-clique that contains it. \vspace{2mm}

 \begin{definition}\label{Def1}
  Let $G$ be an arbitrary graph with $u\in V(G)$. 
   The vertex clique weight is then defined as $$w_G(u)=\sum_{i}\frac{1}{i}k_i(u;G),$$
  where the  summation is taken over all $i$-cliques, $i\ge 1$, in $G$ that contain the vertex $u$.
\end{definition}


\begin{lemma}\label{K_r+1}
  Let $G$ be a graph with $u\in V(G)$. If $d_G(u)=r ~(r\ge 1)$, then $$ w_G(u)\le \frac{1}{r+1}(2^{r+1}-1),$$  with equality  if and only if $G[N(u)]=K_r$.
\end{lemma}
\begin{proof} Given $d_G(u)=r$, we observe that  $k_t(u)=0 $ for $t>r+1$ and $k_t(u)\le \binom{r}{t-1}$ for $t\in [r+1]$.
 Equality holds if and only if  $G[N(u)]=K_r$. By Definition \ref{Def1},
$$w_G(u)=\sum_t\frac{1}{t}k_t(u)\le \sum_{t=1}^{r+1}\frac{1}{t}\binom{r}{t-1}=\frac{1}{r+1}(2^{r+1}-1),$$  where the equality holds if and only if $G[N(u)]=K_r$.
Consequently, the proof is complete.
\end{proof}

\begin{theorem}\label{second}
	Let $G$ be an arbitrary graph. Then
	$$k(G)=1+\sum_{u\in V(G)}w_G(u).$$
\end{theorem}
\begin{proof}
	The result follows directly from Definition \ref{Def1} and the observation that the empty clique is  the only clique that contains no vertices.
\end{proof}	

Below, we show  the clique counting formula in Theorem 2.4 also satisfies the principle of inclusion-exclusion.

\begin{theorem}\label{first}
Let $ G$ be the union of graphs  $G_1$ and $ G_2 $ with intersection
 $ G_1 \cap G_2=H$.  If  $G$ contains no clique containing edges from both of $E(G_1)\backslash E(H)$ and $E(G_2)\backslash E(H)$ and $u\in V(H)$,
 then
$$w_G(u) = w_{G_1}(u) + w_{G_2}(u) - w_H(u).$$
\end{theorem}	
\begin{proof} Given  $ G_1 \cap G_2=H$, there are no  cliques containing vertices
from both $  V(G_1) \backslash V(H) $ and  $  V(G_2) \backslash V(H) $,
and then every clique is entirely contained within  either $G_1$ or $G_2$. Thus, for each $i$, we have
$k_i(u;G)=k_i(u;G_1) +k_i(u;G_2) - k_i(u;H).$
Consequently, 	
\begin{align*}
w_G(u)  & =\sum_i \frac{1}{i} k_i(u;G) \\
&  = \sum_i \frac{1}{i} (k_i(u;G_1) +k_i(u;G_2) - k_i(u;H)) \\
& = w_{G_1}(u) + w_{G_2}(u) - w_H(u).
\end{align*}
This  completes the proof.		
\end{proof}
As a direct consequence of Theorem \ref{first}, we obtain the following result:

\begin{theorem}\label{third}
	Let $G$ be the union of graphs $G_1$ and $G_2$ with intersection $G_1 \cap G_2 = H$. If $G$ contains no clique containing edges from both of $E(G_1)\backslash E(H)$ and $E(G_2)\backslash E(H)$, then
	$$k(G) = k(G_1) + k(G_2) - k(H).$$
\end{theorem}

\begin{proof}
	From Definition \ref{Def1}, for each vertex $u$:\\
	- If $u \in V(H)$, then $k_i(u;G)=k_i(u;G_1) +k_i(u;G_2) - k_i(u;H)$.\\
	- If $u \in V(G_1) \setminus V(H)$, then $k_i(u;G) = k_i(u;G_1)$.\\
	- If $u \in V(G_2) \setminus V(H)$, then $k_i(u;G) = k_i(u;G_2)$.
	
	Thus,
	\begin{align*}
	k(G) & = \sum_{u \in V(G)} w_G(u) + 1 \\
	& = \sum_{u \in V(G_1) \setminus V(H)} w_{G_1}(u) + \sum_{u \in V(G_2)
		 \setminus V(H)} w_{G_2}(u) + \sum_{u \in V(H)} (w_{G_1}(u) + w_{G_2}(u) - w_H(u)) + 1 \\
	& = k(G_1) + k(G_2) - k(H).
	\end{align*}
	This completes the proof.
\end{proof}

\noindent {\bf Remark 1.} In Theorem \ref{third}, the condition that $G$ contains no clique containing edges from both of $E(G_1)\backslash E(H)$ and $E(G_2)\backslash E(H)$ if necessary. Let $G=(V,E)$ with $V=\{1,2,3\}$ and $E=\{12,13,23\}$.
Suppose that $G_1=(V_1,E_1)$ with $V_1=\{1,2\}$ and $E_1=\{12\}$, and $G_2=(V_2,E_2)$ with $V_2=\{1,2,3\}$ and $E_1=\{13,23\}$. Then $k(G) \neq  k(G_1) + k(G_2) - k(G_1\cap G_2).$


\begin{example} Let $G$ be the  graph depicted in Figure 1. Define $G_1:=G[\{1,2,3,4\}]$, $G_2:=G[\{3,4,5,6\}]$
and $H:=G[\{3,4\}]$. It can be verified that
$G=G_1\cup G_2$ and  $H=G_1\cap G_2$. Then $G$ contains no clique containing edges from both of $E(G_1)\backslash E(H)$ and $E(G_2)\backslash E(H)$. Therefore,
$$w_G(1)=w_{G_1}(1)=\frac{15}{4}, ~~w_G(3)=w_{G_1}(3)+w_{G_2}(3)-w_{H}(3)=\frac{15}{4}+\frac{19}{6}-\frac{3}{2}=\frac{65}{12},~~w_G(5)=w_{G_2}(5)=\frac{7}{3}$$
and
$$k(G)=k(G_1)+k(G_2)-k(H)=16+12-4=24.$$

\vspace{5mm}

\begin{tikzpicture}

\coordinate (A) at (0,0);
\coordinate (B) at (0,2);
\coordinate (C) at (1,1);
\draw (A) -- (B) -- (C) -- cycle;
\coordinate (D) at (2,1);
\draw (C) -- (D);
\draw (A) -- (D);
\draw (B) -- (D);
\coordinate (E) at (3,2);
\coordinate (F) at (3,0);
\draw (C) -- (E);
\draw (C) -- (F);
\draw (D) -- (E);
\draw (D) -- (F);
\foreach \point in {A,B,C,D,E,F}
    \fill (\point) circle (2pt);
\node at (A) [below] {$1$};
\node at (B) [above] {$2$};
\node at (C) [below] {$3$};
\node at (D) [below] {$4$};
\node at (E) [below] {$5$};
\node at (F) [below] {$6$};

\coordinate (A) at (7,0);
\coordinate (B) at (7,2);
\coordinate (C) at (8,1);
\draw (A) -- (B) -- (C) -- cycle;
\coordinate (D) at (9,1);
\draw (C) -- (D);
\draw (A) -- (D);
\draw (B) -- (D);
\foreach \point in {A,B,C,D}
    \fill (\point) circle (2pt);
\node at (A) [below] {$1$};
\node at (B) [above] {$2$};
\node at (C) [below] {$3$};
\node at (D) [below] {$4$};

\coordinate (C) at (11,1);
\coordinate (D) at (12,1);
\coordinate (E) at (13,2);
\coordinate (F) at (13,0);
\draw (C) -- (D);
\draw (C) -- (E);
\draw (C) -- (F);
\draw (D) -- (E);
\draw (D) -- (F);
\foreach \point in {C,D,E,F}
    \fill (\point) circle (2pt);
\node at (C) [below] {$3$};
\node at (D) [below] {$4$};
\node at (E) [below] {$5$};
\node at (F) [below] {$6$};

\node at (1.5, -1) {G};
\node at (8, -1) {$G_1$};
\node at (12, -1) {$G_2$};

\node at (7, -2) {Figure ~1.};

\end{tikzpicture}

\end{example}


\begin{example}
Let $m$, $r$, and $p$  be integers satisfying $m=\binom{r}{2}+p+1$ and $0\le p<r-1$.
 The colex graph $\mathcal{C}(r+1,m)$ can be considered as the union of two graphs $K_r$ and $K_{p+2}$ such that their intersection  is $K_{p+1}$.
 Then, applying Theorem \ref{first}, we find
 $$w_{\mathcal{C}(r+1,m)}(\Delta)=\max \{w_{\mathcal{C}(r+1,m)}(v) | v\in V(\mathcal{C}(r+1,m))\}=\frac{2^r-1}{r}+\frac{2^{p+2}-1}{p+2}-\frac{2^{p+1}-1}{p+1}.$$
Furthermore, by Theorem  \ref{third}, we have
 $$ k(\mathcal{C}(r+1,m))=2^r+2^{p+2}-2^{p+1}.$$
\end{example}

\begin{example}
Let $r$ and $t$ be integers with $r\ge 2t$ and $t\ge 1$. $K_r-tK_2$ can be considered as the union of two copies of $K_{r-1}-(t-1)K_2$ such that their intersection  is $K_{r-2}-(t-1)K_2$.
Applying Theorem \ref{third} again, we have
$$k(K_r-tK_2) = 2k(K_{r-1}-(t-1)K_2) -k(K_{r-2}-(t-1)K_2).$$
\end{example}


%
Next, we introduce the concept $k(e;G)$, which denotes the number of cliques in $G$ that contain the edge $e$. In contrast to $w_G(u)$, we refer it as  the clique weight with respect to a given  edge.
 \begin{definition}\label{Def2}
  Let $G$ be an arbitrary graph and $e\in E(G)$. We use  $k(e;G)$ to denote the number of cliques in $G$ that contain the edge $e$.
  For simplicity, we write $k(e;G)$ as $k(e)$ if there is no ambiguity.
\end{definition}

As is well known, for an arbitrary graph $G$ and an edge $e\in E(G)$, all cliques contained in $G$
 can be partitioned into two parts: those containing the edge $e$ and those not containing the edge $e$.
  The cardinality of the former is $k(e;G)$ and the cardinality of the latter is $k(G-e)$. Therefore, we have
\begin{theorem}\label{forth}
	 Let $G$ be an arbitrary graph and $e\in E(G)$. Then
	$$k(G) = k(G-e) + k(e;G).$$
\end{theorem}
%

To conclude this section, we introduce two lemmas related to the clique weight of edges. These lemmas will play a crucial role in the proof presented in the final chapter.
\begin{lemma}\label{key3}
	Let $\Gamma$ be the graph obtained from two disjoint complete graphs $P:=K_{2p}$ and $Q:=K_{q}$ by adding  $2p$ edges between $P$ and $Q$, where $p\ge 1$ and $q\ge 1$.
Suppose that $d_P(w)\le p$ and $d_Q(u)\le p$ hold for any $w\in V(Q)$ and $u\in V(P)$. Then
$$k(e;\Gamma)\le 2^p$$
for any edge  $e$ between $P$ and $Q$.
\end{lemma}
\begin{proof}
Consider an edge  $e=uw\in E(\Gamma)$, where $u\in V(P)$ and $w\in V(Q)$.
Let $F$ be a maximal clique containing the edge $e$. 
Then the union of the vertex set $\{u,w\}$ and an arbitrary subset of $V(F)\backslash \{u,w\}$ induces a clique containing the edge $e$.
 Therefore, the number of those cliques that are subsets of $F$ and contain the edge $e$ is exactly
  $$2^{|V(F)|-2}.$$

Define $U:=N_{P}(w)$ and $W:=N_{Q}(u)$. When $\min \{|W|,|U|\}=1$, for instance, $|W|=1$,
the subgraph  $\Gamma[U\cup W]$ forms the unique maximal clique of $\Gamma$ that contains the edge $e$. In this case, $\Gamma[U\cup W]$ is complete and thus
 $k(e;\Gamma)$ is given by
 $2^{|U\cup W|-2}$. Since $|U|\le p$, it follows that $k(e;\Gamma)\le 2^{p-1}.$ This completes the proof.\vspace{2mm}

In the subsequent discussion, we always assume that  $\min \{|W|,|U|\}\ge 2$.
It is readily observed that both $\Gamma[U\cup \{w\}]$ and $\Gamma[W\cup \{u\}]$ are complete and thus both of them are cliques  containing the edge $e$.
For convenience, we define  $F_1:=\Gamma[U\cup \{w\}]$ and $F_2:=\Gamma[W\cup \{u\}]$.
We partition our discussion into the following three cases:\vspace{2mm}

\noindent {\bf Case 1.} $F_1$ and $F_2$ are the only two maximal cliques of $\Gamma$ containing the edge $e$.

\noindent As established in the preceding discussion, there are precisely $2^{|U\cup \{w\}|-2}$
cliques contained within $F_1$
that include the edge $e$, and similarly, exactly $2^{|W\cup \{u\}|-2}$
cliques contained within $F_2$
that include the edge $e$. Given that  $|U|\le p$ and $|W|\le p$, applying the principle of inclusion-exclusion,
we can deduce that $k(e;\Gamma)\le 2^{|U|-1}+2^{|W|-1}-2^{|V(F_1\cap F_2)|-2}< 2^p.$ This completes the proof.\vspace{2mm}

\noindent {\bf Case 2.} At least one of $F_1$ and $F_2$ is not the maximal clique of $\Gamma$ that contains the edge $e$.

\noindent Assume without loss of generality that $F_1$  is not  maximal. Then there exists a maximal clique, denoted by $F_3$, such that $V(F_3)\cap U=U$ and
$\{w\}\subsetneq V(F_3)\cap W$. Set $W^*=V(F_3)\cap W$. Given that there are $2p$ edges between $P$ and $Q$, it follows that
$|U||W^*|\le 2p$, which implies that $2^{|U|+|W^*|-2}\le 2^p$. Equality holds if and only if $\min \{|U|,|W^*|\}= 2$, $\max  \{|U|,|W^*|\}= p$, and $W^*=W$.
Applying the principle of inclusion-exclusion, we have
\begin{align*}
k(e;\Gamma)&=2^{|V(F_3)|-2}+2^{|V(F_2)|-2}-2^{|V(F_3\cap F_2)|-2}\\
&=2^{|U|+|W^*|-2}+2^{|W|-1}-2^{|W^*|-1}\\
&\le 2^p.
\end{align*}
Thus the result follows in such a case.\vspace{2mm}


\noindent {\bf Case 3.} There exist at least three maximal cliques of $\Gamma$  that contain the edge $e$.

Obviously, both of $F_1$ and $F_2$ are maximal clique of $\Gamma$  that contains the edge $e$.
Let now $F_4$ be another maximal clique of $\Gamma$, distinct from  $F_1$ and $F_2$,  that contains the edge $e$.
Define $U^{**}:=V(F_4)\cap U$ and $W^{**}:=V(F_4)\cap W$. Because $F_1$, $F_2$ and $F_4$ are all maximal cliques,
we have $\{u\}\subsetneq U^{**} \subsetneq U$ and $\{w\}\subsetneq W^{**} \subsetneq W$. Therefore, we have $\max  \{|U|,|W|\}\ge 3$,
$p\ge 3$, $\min \{|U^{**}|,|W^{**}|\}\ge 2$ and $|U^{**}||W^{**}|\le 2p-2$.
%
Then, combining with the principle of inclusion-exclusion and the fact that $V(F_1\cap F_2)=\{u,w\}$, we have
\begin{align*}
k(e;\Gamma)&=2^{|V(F_1)|-2}+2^{|V(F_2)|-2}+2^{|V(F_4)|-2}-2^{|V(F_1\cap F_4)|-2}-2^{|V(F_2\cap F_4)|-2}\\
&-2^{|V(F_1\cap F_2)|-2}+2^{|V(F_1\cap F_2\cap F_4)|-2}\\
&\le 2^{|U|-1}+2^{|W|-1}+2^{|U^{**}|+|W^{**}|-2}-2^{|U^{**}|-1}-2^{|W^{**}|-1}\\
&\le 2^p.
\end{align*}
(Given that $|U^{**}||W^{**}|\le 2p-2$, it follows that $|U^{**}|+|W^{**}|-2\le p-1$. Consequently, the desired result is achieved if at least two of the quantities
 $|U|-1$, $|V|-1$ and $|U^{**}|+|W^{**}|-2$ is no more than $p-1$, since the maximum value among them is no more than $p-1$.
Therefore, we need at least two of these quantities to be equal to $p-1$, for instance, $|U|=|V|=p$ and $|U^{**}|=|W^{**}|=2$. The inequality still holds.)
Thus the result follows in such a case.\vspace{2mm}

Up to now, the proof is complete.
\end{proof}

\begin{lemma}\label{key4}
	Let $\Gamma$ be the graph obtained from two disjoint graphs $P:=K_{2p}-tK_2$ and $Q:=K_{q}$ by adding  $2p$ edges between $P$ and $Q$, where $p$, $t$ and $q$ are positive integer numbers with $p\ge t$.
	Suppose that $N_P(w)=\{u_i|i=1,2,\ldots,p+t\}$ such that  $u_{2i-1}u_{2i}\notin E(P)$ for each $i\in [t]$. Denote by $k(\cup_{i=1}^{2t}e_i;\Gamma)$ the number of cliques of $\Gamma$ including at least one edge of $\sum_{i=1}^{2t}e_i$. Then
	$$k(\cup_{i=1}^{2t}e_i;\Gamma)\le 2^{p-t}+k(K_{p+t+1}-tK_2)-k(K_{p+t}-tK_2).$$
\end{lemma}
\begin{proof}
  We  partition $k(\cup_{i=1}^{2t}e_i;\Gamma)$ into two types: 1) cliques containing
  at least one edge $u_iw$, for some $i$, and at least one vertices from $N_{Q}(u_i) \backslash \{w\}$; and 2)
 cliques containing at least one edge $u_iw$, for some $i$, and no vertices in $N_{Q}(u_i) \backslash \{w\}$.
  Given that there are $2p$ edges between $P$ and $Q$, and $d_P(w)=p+t$, it follows that $\sum_{i=1}^{2t}d_Q(u_i)\le p+t$.
  For $i=1,2,\ldots,2t$, denote by $H_i$ the maximum clique containing
  the edge $u_iw$ and at least one vertices from $N_{Q}(u_i) \backslash \{w\}$.
  Obviously, $H_i=\emptyset$ if $d_Q(u_i)=1$, and we can concentrate the neighbors of vertices $u_i$ and $u_j$
 in $Q$ onto one of the vertices, thereby increasing the sum of the number of cliques of $H_i$ and $H_j$
if $d_Q(u_i)>1$ and $d_Q(u_j)>1$.
  Therefore, to maximize  the first type, we can set $d_Q(u_i)=1$ for $i=1,2,\ldots,2t-1$, and $d_Q(u_{2t})=p-t+1.$
 Consequently, the sum of those cliques coming from the first type is bounded by $2^{p-t}$.
 Since each clique containing the edge $e_i$ and no vertices in $N_{Q}(u_i) \backslash \{w\}$ is a subgraph of
  $\Gamma[N_P(w)\cup \{w\}]$,
  the sum of those cliques coming from the second type is bounded by $k(\Gamma[N_P(w)\cup \{w\}])-k(\Gamma[N_P(w)])$.
 Consequently,
 $$k(\cup_{i=1}^{2t}e_i;\Gamma)\le 2^{p-t}+k(K_{p+t+1}-tK_2)-k(K_{p+t}-tK_2)$$ since $\Gamma[N_P(w)\cup \{w\}]=K_{p+t+1}-tK_2$ and $\Gamma[N_P(w)]=K_{p+t}-tK_2$.
\end{proof}

\section{Proof of Theorem \ref{maina}}
In this section, we will complete the proof of Theorem \ref{maina} and Theorem \ref{main0}.
As is well known that Theorem \ref{maina} is a generalization of Theorem \ref{main0}. Therefore, we below focus on the proof of
 Theorem \ref{maina}.
 Firstly, we need to establish a lemma.
\begin{lemma}\label{key11}
	Let $r$, $s$ and $p$ be fixed integer numbers with $r\ge 3$  and $0\le s\le r$. Then the inequality $$2^{r+1}+2^s-1> (r+1+s)(\frac{2^r-1}{r}+\frac{2^{p+2}-1}{p+2}-\frac{2^{p+1}-1}{p+1})$$
	holds if one of  the following  conditions is satisfied:\\
	{\em 1)} $r\geq12$ and $ 0\le p\le \frac{2r}{3}$;\\
	{\em 2)} $7\le r\le 11$ and $0\le p\le r-4$;\\
	{\em 3)} $4\le r\le 6$ and $0\le p\le r-3$;\\
 {\em 4)} $r=3$ and $0\le p\le 1$,  except for the case where $r=3$, $s=2$ and $p=1$.
	\end{lemma}
\begin{proof}
		1)
		 For this case, we prove the statement by considering two subcases, namely \(s = r\) and \(s < r\).
		 First, assume that \(s = r\),
		 we define the function $h(s,p,r)$ as follows: $$h(s,p,r)=2^{r+1}+2^s-1- (r+1+s)(\frac{2^r-1}{r}+\frac{2^{p+2}-1}{p+2}-\frac{2^{p+1}-1}{p+1}).$$
		 When $r\ge 12,$  $s=r$ and $0\leq p\le \frac{2r}{3},$ one can verify that $$h(s,p,r)\geq h(12,p,12)>0.$$
		Next, consider the subcase \(s < r\).
		Note that
		$$\begin{aligned}
			\frac{2^{r}-1}{r} + \frac{2^{p+2}-1}{p+2} - \frac{2^{p+1}-1}{p+1}
			&<   \frac{2^{r}}{r} + \frac{p \cdot 2^{p+1}}{(p+1)(p+2)} + \frac{1}{(p+1)(p+2)}  \\
			&<  \frac{2^{r}}{r} + \frac{ 2^{p+1}}{p} + \frac{1}{2}  \\
			&\le \frac{2^{r+1}+3 \cdot 2^{\frac{2r}{3}+1}+r}{2r},
		\end{aligned}
		$$
		then it suffices to show that
		$$2r(2^{r+1} + 2^s -1) > (r+1+s) \left( 2^{r+1}+3 \cdot 2^{\frac{2r}{3}+1}+r \right).$$
		We define the function $f(s)$ on $s$ as follows:
		$$	\begin{aligned}
			f(s) &= 2r \left( 2^{r+1} + 2^{s}-1 \right) - \left( r+1+s \right) \left( 2^{r+1} + 3 \cdot 2^{\frac{2r}{3}+1} + r \right) \\
			&= r \cdot 2^{r+1} + r \cdot 2^{s+1} - 3r \cdot 2^{\frac{2r}{3}+1} - (s+1) \left( 2^{r+1} + 3 \cdot 2^{\frac{2r}{3}+1} \right) - r \left( r+3+s \right).
		\end{aligned}
		$$
		One can verify that
		$$	\begin{aligned}
			f(r-1) &= r \cdot 2^{r+1} + r \cdot 2^{r} - 3r \cdot 2^{\frac{2r}{3}+1} - r \left( 2^{r+1} + 3 \cdot 2^{\frac{2r}{3}+1} \right) - 2r \left( r+1 \right)\\
			&=  r \cdot 2^{\frac{2r}{3}+1}(2^{\frac{r}{3}} - 6) - 2r \left( r+1 \right)\\
			&>0,
		\end{aligned}
		$$where $f(r-1)$ is an  increasing function on $r$.
		Similarly, $f(s)>0$ when $r\ge 12$ and $s\in\{r-2,r-3,r-4\}.$
		For $s\le r-5$, we have
		$$	\begin{aligned}
			f(s) &= 2r \left( 2^{r+1} + 2^{s}-1 \right) - \left( r+1+s \right) \left( 2^{r+1} + 3 \cdot 2^{\frac{2r}{3}+1} + r \right) \\
			&\geq 2r \cdot 2^{r+1}   - \left( r+1+s \right) \left( 2^{r+1} + 3 \cdot 2^{\frac{2r}{3}+1} + r \right) \\
			&\geq 2r \cdot 2^{r+1}   - \left( r+1+r-5 \right) \left( 2^{r+1} + 3 \cdot 2^{\frac{2r}{3}+1} + r \right) \\
			&\geq 2^{r+3}- 3(2r-4) \cdot 2^{\frac{2r}{3}+1}-2r(r-2)\\
			&>0,
		\end{aligned}
		$$
		where $g(r)=2^{r+3}- 3(2r-4) \cdot 2^{\frac{2r}{3}+1}-2r(r-2)$ is an  increasing function on $r(r\geq12)$.

		Consequently, 1) follows.
		
		2) First, assume that \(s = r\),
		when $7\leq r\leq 11$ and $ 0\leq p\le r-4,$ one can verify that $h(s,p,r)\geq h(7,p,7)>0.$
		
		Next, consider the subcase \(s < r\).
		Similarly, we have
		$$\begin{aligned}
			 \frac{2^{r}-1}{r} + \frac{2^{p+2}-1}{p+2} - \frac{2^{p+1}-1}{p+1}
			& < \frac{2^{r}}{r} + \frac{p \cdot 2^{p+1}}{(p+1)(p+2)} + \frac{1}{(p+1)(p+2)} \\
			& < \frac{2^{r}}{r} + \frac{(r-4) \cdot 2^{r-3}}{(r-3)(r-2)} + \frac{1}{(r-3)(r-2)},
		\end{aligned} $$
	where $k(p)=\frac{p \cdot 2^{p+1}}{(p+1)(p+2)} + \frac{1}{(p+1)(p+2)}$ is an  increasing function on $p$ when $7\leq r\leq 11$ and $ 0\leq p\le r-4.$
Then it suffices to show that
		$$r \left( r - 3 \right) \left( r - 2 \right) \left( 2^{r +1} - 1 + 2^s \right)  > \left( r + 1 + s \right) \left( 2^r \left( r - 3 \right) \left( r - 2 \right) + r \left( r - 4 \right) \cdot 2^{r-3}  +r \right)$$
		when $7\le r\le 11$ and $s<r$.
Let
		$$	
			f(s) = r \left( r - 3 \right) \left( r - 2 \right) \left( 2^{r +1} - 1 + 2^s \right)  - \left( r + 1 + s \right) \left( 2^r \left( r - 3 \right) \left( r - 2 \right) + r \left( r - 4 \right) \cdot 2^{r-3}  +r \right)
		$$ be a function $f(s)$ on $s$.
		By a direct verification, we get $f(s)>0$ when $7\leq r\leq 11$ and $s\in\{r-1,r-2,r-3\}.$
              For $s\le r-4$, we have
       $$	\begin{aligned}
       	f(s) &\ge  r \left( r - 3 \right) \left( r - 2 \right)  2^{r +1}   - \left( r + 1 + s \right) \left( 2^r \left( r - 3 \right) \left( r - 2 \right) + r \left( r - 4 \right) \cdot 2^{r-3}  +r \right) \\
         	& \geq 2^{r+1} (r - 3) (r - 2) r - \left( 2^r (r - 3) (r - 2) + 2^{r-3} (r - 4) r + r \right) (2r-3)\\
          	&>0,
       \end{aligned}
       $$
       where $q(r)=2^{r+1} (r - 3) (r - 2) r - \left( 2^r (r - 3) (r - 2) + 2^{r-3} (r - 4) r + r \right) (2r-3)$ is an  increase function on $r(7\leq r\leq 11)$.

       Consequently, 2) follows.

       3)-4) The results follow by a direct verification and the details are omitted.
	\end{proof}	

\noindent{\bf Remark 2.} If $r=3$, $s=2$ and $p=1$, then $2^{r+1}+2^s-1= (r+1+s)(\frac{2^r-1}{r}+\frac{2^{p+2}-1}{p+2}-\frac{2^{p+1}-1}{p+1})$.

\noindent \begin{proof}[\bf Proof of Theorem~\ref{maina}.]
Let $G$ be the  graph
  maximizing the number of cliques among  graphs with order $n$ and maximum degree at most $r$.
 We divide our proof into the following three cases:\\
{\bf Case 1.} $r=1$. According to Definition \ref{Def1},
 $w_G(u)=\frac{3}{2}$ if $d(u)=1$,  and $w_G(u)=1$ if $d(u)=0$, which implies  that
 $G=aK_2\cup K_b$ uniquely maximizes the number of cliques among  graphs with maximum degree at most $1$.\vspace{2mm}

\noindent {\bf Case 2.} $r=2$. Let $d(u)=2$. According to Definition \ref{Def1},  $w_G(u)=\frac{7}{3}$ if $u$ is contained in some triangle and $w_G(u)=2$ otherwise.
Moreover, since both values dominate the clique weight of those vertices whose degree is less than $2$, $k(C_4)=k(C_3\cup K_1)$ and $k(C_5)=k(C_3\cup K_2)$,
 maximality requires maximizing triangle formation and thus
 $G$ is the unique extremal graph, where
 $G=\frac{n}{3}C_3$ if $n\equiv 0\pmod{3}$, and $G=(\lfloor\frac{n}{3}\rfloor-1)C_3\cup C_4$ or $\lfloor\frac{n}{3}\rfloor C_3\cup K_1$ if $n\equiv1\pmod{3}$, and
  $G=(\lfloor\frac{n}{3}\rfloor-1)C_3\cup C_{5}$ or $\lfloor\frac{n}{3}\rfloor C_3\cup K_2$  otherwise.\vspace{2mm}


   \noindent {\bf Case 3.} $r\ge 3$.
  In view of $k(aK_{r+1}\cup K_b)=a(2^{r+1}-1)+2^b$,
  then $k(G)\ge a(2^{r+1}-1)+2^b.$
 Set
   $$w_G(\Delta)=\max \{w_G(v) | v\in V(G)\}.$$
We will show that $w_G(\Delta)= \frac{2^{r+1}-1}{r+1}$ through three assertions.

 \noindent {\bf Assertion 1.}
  \begin{align}
w_G(\Delta)>
\begin{cases}
\frac{2^r-1}{r}+\frac{2^{\frac{2r}{3}+2}-1}{\frac{2r}{3}+2}-\frac{2^{\frac{2r}{3}+1}-1}{\frac{2r}{3}+1} &~ {\rm if}~ r\ge 12; \\
\\
\frac{2^r-1}{r}+\frac{2^{r-2}-1}{r-2}-\frac{2^{r-3}-1}{r-3} &~ {\rm if}~ 7\le r\le 11; \\
\\
\frac{2^r-1}{r}+\frac{2^{r-1}-1}{r-1}-\frac{2^{r-2}-1}{r-2} &~ {\rm if}~ 4\le r\le 6; \\
\\
\frac{2^r-1}{r}+\frac{2^{1+2}-1}{1+2}-\frac{2^{1+1}-1}{1+1}=\frac{19}{6} &~ {\rm if}~ r=3.
\end{cases}
\end{align}

  For $r\ge 12$,  assume to the contrary that $w_G(\Delta)\le \frac{2^r-1}{r}+\frac{2^{\frac{2r}{3}+2}-1}{\frac{2r}{3}+2}-\frac{2^{\frac{2r}{3}+1}-1}{\frac{2r}{3}+1}$.
  We set $\frac{2^r-1}{r}+\frac{2^{\frac{2r}{3}+2}-1}{\frac{2r}{3}+2}-\frac{2^{\frac{2r}{3}+1}-1}{\frac{2r}{3}+1}:=A$
  for simplicity. Applying Lemma \ref{key11}, we have
\begin{align*}
		k(G) & \le nA \\
& = (a-1)(r+1)A+(r+1+b)A\\
		& < (a-1)(r+1)\frac{2^{r+1}-1}{r+1}+(2^{r+1}-1)+2^b\\
		& <a(2^{r+1}-1)+2^b,
		\end{align*}
which yields a contradiction. The proofs for other cases are similar and are omitted for brevity.
We should point out  that, when $r=3,s=2$ and $p=1$ described in Lemma \ref{key11}, the condition that $w_G(\Delta) >\frac{2^r-1}{r}+
\frac{2^{1+2}-1}{1+2}-\frac{2^{1+1}-1}{1+1}=\frac{19}{6}$ is essential.
 This is because, in such a case, not every vertex possesses an identical clique weight,
even though the equation $(r+1+b)w_G(\Delta)=k(K_{r+1}\cup K_b)$ holds.

 \noindent {\bf Assertion 2.} Let $u$ be the vertex with $w_G(u)=w_G(\Delta)$ and let $p$ be the integer
 number such that $|E(G[N[u]])|=\binom{r+1}{2}-p$. Then
  \begin{align}
 p\le
 \begin{cases}
 \lfloor\frac{r}{3}\rfloor-1 &~ {\rm if}~ r\ge 12; \\
 2 &~ {\rm if}~ 7\le r\le 11; \\
 1 &~ {\rm if}~ 4\le r\le 6; \\
 0 &~ {\rm if}~ r=3.
 \end{cases} ~~~~~~~~~
 \end{align}

For $r\ge 12$, from Assertion 1, $w_{G}(\Delta)> \frac{2^r-1}{r}+\frac{2^{\frac{2r}{3}+2}-1}{\frac{2r}{3}+2}-\frac{2^{\frac{2r}{3}+1}-1}{\frac{2r}{3}+1}$.
Then, by Theorem \ref{lemma1}, we have
 $$w_{G}(\Delta)\ge w_{\mathcal{C}(r+1,\binom{r}{2}+\lceil
 	\frac{2r}{3}\rceil+1)}(\Delta),$$ here we use $\lceil
 	\frac{2r}{3}\rceil$ to ensure the integrity and thus making the inequality weaker.
 Applying Theorem \ref{lemma1} again, we have
%
  \begin{align*}
|E(G[N[u]])| &\ge |E(\mathcal{C}(r+1, \binom{r}{2} + \left\lceil \frac{2r}{3} \right\rceil + 1))| \\
&= \binom{r}{2} + \left\lceil \frac{2r}{3} \right\rceil + 1 \\
&= \binom{r+1}{2} - r + 1 + \left\lceil \frac{2r}{3} \right\rceil \\
&= \binom{r+1}{2} - \left\lfloor \frac{r}{3} \right\rfloor + 1.
\end{align*}
 The proofs for other cases are similar and are omitted for brevity.

%
%
%
%
%

\noindent {\bf Assertion 3.}  $w_G(\Delta)=\frac{2^{r+1}-1}{r+1}$.
 When $n=r+1$, by Lemma \ref{K_r+1} $K_{r+1}$ is the unique graph satisfying $n\cdot w_{K_{r+1}}(\Delta)\ge 2^{r+1}-1$.
When $r=3$, to ensure that $w_G(\Delta)$ satisfies Inequality (1), $G$ must be $aK_4\cup K_b$  by Assertion 2.
 In this case,  $w_G(\Delta)=\frac{2^{3+1}-1}{3+1}$, and thus the result holds when $n=r+1$ or $r=3$.
 Therefore, in the following discussion, we always assume that $n>r+1$ and $r \ge 4.$ \vspace{2mm}
	
Assume to the contrary that $w_G(u)=w_G(\Delta)<\frac{2^{r+1}-1}{r+1}$.	
 Then, combining with Assertion 2 and Lemma \ref{K_r+1}, $G[N[u]]\neq K_{r+1}$ and thus $|E(G[N[u]])|= \binom{r+1}{2}-p$, where $p$ satisfies the Inequality (1).
  For the sake of convenience, let $H:=G[N[u]]$,
 $N(u)=\{u_i|i=1,2,\ldots,r\}$,  $V(G)\backslash N[u]=\{w_i|i=1,2,\ldots,n-r-1\}$ and
 $E(H,G\backslash H)=\{e_i|i=1,2,\ldots,s\}$. Then one of the following two subcases must  occur:
  \vspace{2mm}

\noindent{\bf Subcase 3.1.}  For each $i(i=1,2,\ldots,n-r-1)$,  $d_H(w_i)\le p$.

Let $e_i=u_iw_i$ for $i=1,2$ such that $u_1u_2\notin E(H)$. Note that $w_1$ and $w_2$
may be the same vertex.
 Let $G^*$ be the graph obtained from $G$ by deleting the edges $e_1$ and $e_2$. Then $$k(G)-k(G^*)\le k(e_1;G)+k(e_2;G)\le 2^{p+1},$$ where the last inequality holds  from Lemma \ref{key3},
  since $G$ contains a  subgraph which is a subgraph of $\Gamma$ described in Lemma \ref{key3}.
   Let now $G^{**}$ be the graph obtained from $G^*$ by adding the edge $u_1u_2$.
Recall that $\overline{H}$ contains  $p$ edges, which implies that $G[N[u]]$ contains at least $r+1-2p$ $r$-degree vertices.
Therefore, $G^*$ contains a subgraph which  is isomorphic to $K_{r-2p+3}-K_2$ that includes the vertices $u_1$ and $u_2$.
 Consequently,  \begin{align*}
  k(G^{**})-k(G^*) &\ge k(K_{r-2p+3})-k(K_{r-2p+3}-K_2)\\
  &=2^{r-2p+1}\\
  &>2^{p+1},
  		\end{align*}
where the last inequality holds for $r\ge 3p+1$ by Inequality (2), which yields a contradiction to  the maximality of the graph $G$.
     Hence, this case can not occur.\vspace{2mm}

\noindent  {\bf Subcase 3.2.} There exists a vertex, say $w_1$, such that $d_H(w_1)>p$.

\noindent Suppose  that $d_H(w_1)=p+t$ with $t\ge 1$. Set $\{e_i=u_iw_1|i=1,2,\ldots,p+t\}$.
Then there exist at least $t$ pair of vertices, say $(u_{2i-1},u_{2i})$ for $i=1,2,\ldots,t$,
 such that  $u_{2i-1}u_{2i}\notin E(H)$.
 Thus $G[N(w_1)\cup w_1]$ is  isomorphic to a subgraph of $K_{p+t+1}-tK_2$.
 Let $G^*$ be the graph obtained from $G$ by deleting the edges $\{e_i|i=1,2,\ldots,2t\}$.
 Then $G$ contains a subgraph which is a subgraph of $\Gamma$ described in Lemma \ref{key4}, thus, combining this with Lemma \ref{key4}, we have
  \begin{align*}
  k(G)-k(G^*) &\le \sum_{i=1}^{2t}k(\cup_{i=1}^{2t}e_i;G)\\
  &\le  2^{p-t}+ k(G[N(w_1)\cup w_1])-k(G^*[N(w_1)\cup w_1])\\
  &<2^{p-t}+k(K_{p+t+1}-tK_2)-k(K_{p+t}-tK_2),
  		\end{align*}
where $k(\cup_{i=1}^{2t}e_i;G)$ denotes the number of cliques of $G$ including at least one edge of $\sum_{i=1}^{2t}e_i$.

  Let  $G^{**}$ be the graph obtained from $G^*$ by adding the edges $u_{2i-1}u_{2i}$ for $i\in [t]$. Similar to the discussion above,
   $G[N[u]]$ contains at least $r+1-2p$ $r$-degree vertices and thus
  $G[N[u]]$ contains a subgraph of $K_{r+1-2p+2t}-tK_2$ that contains all vertices $\{u_i|i=1,2,\ldots,2t\}$. Therefore, we have
  \begin{align*}
  k(G^{**})-k(G^*) &>   k(K_{r+1-2p+2t})-k(K_{r+1-2p+2t}-tK_2).
  \end{align*}
 Consequently, for $t\ge 2$, we have
  \begin{align*}
 k(G^{**})-k(G) &>  k(K_{r+1-2p+2t})-k(K_{r+1-2p+2t}-tK_2) -k(K_{p+t+1}-tK_2)+k(K_{p+t}-tK_2)-2^{p-t}\\
 &>  k(K_{p+2+2t})-k(K_{p+2+2t}-tK_2) -k(K_{p+t+1}-tK_2)-2^{p-t}~~(by~r\ge 3p+1)\\
 &=  k(K_{p+2+2t}-K_2)+k(K_{p+2t})-k(K_{p+2+2t}-tK_2) -k(K_{p+t+1}-tK_2)-2^{p-t}\\
  &\ge  k(K_{p+2t})-k(K_{p+t+1}-tK_2)-2^{p-t}\\
   &=  k(K_{p+2t}-K_2)+k(K_{p+2t-2})-k(K_{p+t+1}-tK_2)-2^{p-t}\\
    &\ge  k(K_{p+2t-2})-2^{p-t}\\
 & \ge 0.
 \end{align*}
 If $t=1$, then
   \begin{align*}
 k(G^{**})-k(G) &>  k(K_{r-2p+3})-k(K_{r-2p+3}-K_2) -k(K_{p+2}-K_2)+k(K_{p+1}-K_2)-2^{p-1}\\
 &= 2^{r-2p+1}-2^{p+1}\\
   & \ge 2^{p+2}-2^{p+1} ~~(by~r\ge 3p+1)\\
   &>0.
 \end{align*}
Hence, all these results contradict the maximality of the graph $G$.  Hence, this case can not occur.

Up to now, $w_G(\Delta)=\frac{2^{r+1}-1}{r+1}$. Consequently,
 $G$ contains a subgraph  $K_{r+1}$. If the order of the graph $G-K_{r+1}$ is larger  than $r+1$,
  then, by a similar method, $G$ has another $K_{r+1}$. Thus $G$ contains $\lfloor\frac{n}{r+1}\rfloor$ copies of  $K_{r+1}$ and the proof is complete.
\end{proof}

\end{document}